\documentclass{article}
\usepackage{amsmath, amsthm, amssymb}
\usepackage{enumitem}
\usepackage{array, booktabs}
\usepackage{float}
\usepackage{mathtools}
\usepackage{xcolor}
\usepackage{caption}

\usepackage[labelsep=period]{caption}

\theoremstyle{plain}
\newtheorem{thm}{Theorem}[section]
\newtheorem{prop}[thm]{Proposition}

\newtheorem{cor}[thm]{Corollary}
\newtheorem*{clm}{Claim}
\newtheorem*{lem*}{Lemma}

\theoremstyle{definition}
\newtheorem{dfn}[thm]{Definition}
\newtheorem{rem}[thm]{Remark}
\newtheorem{prb}[thm]{Problem}
\newtheorem{ex}[thm]{Example}

\newcommand{\Zn}[1]{\mathbb{Z}/#1\mathbb{Z}}
\newcommand{\regp}[1]{#1_{\mathrm{reg}}}

\title{On the fundamental group of the regular part of Fujiki's compact K\"{a}hler symplectic orbifolds}
\author{Shun Yamaguchi}

\begin{document}
    \maketitle
    \begin{abstract}
        We calculate the fundamental group of the regular part of certain compact K\"{a}hler symplectic orbifolds constructed by Fujiki, called Fujiki's examples.
        We determine which one is an irreducible symplectic orbifold among Fujiki's examples.
        This answers a question posed in \cite{Pe20}.
    \end{abstract}
    
    \section{Introduction}\label{int}
    \subsection*{Symplectic orbifolds and related notions}
    First, we define symplectic orbifolds and related notions.
    Here, we define an \textit{orbifold} to be a normal complex space with only quotient singularities.

    \begin{dfn}
        Let $X$ be a compact K\"{a}hler orbifold and $\omega$ a holomorphic 2-form defined on the regular part $X_\mathrm{reg}$.
        We call the pair $(X, \omega)$ a \textit{symplectic orbifold} when $\omega$ is non-degenerate at each point on $X_\mathrm{reg}$.
    \end{dfn}
    \begin{dfn}
        Let $(X,\omega)$ be a symplectic orbifold.
        \begin{enumerate}[label=(\roman*)]
            \item $(X,\omega)$ is said to be \textit{primitive} when $h^{2,0}(X)=1$.
            \item $(X,\omega)$ is said to be \textit{irreducible} when $h^{2,0}(X)=1$ and $X_\mathrm{reg}$ is simply connected.
        \end{enumerate}
    \end{dfn}

    Irreducible symplectic orbifolds appear in the orbifold version of Beauville-Bogomolov decomposition theorem:
    \begin{thm}[cf. {\cite[Thm. 6.4]{Ca04}}]
        Let $X$ be a compact K\"{a}hler orbifold with trivial first Chern class.
        Then there is a finite, \'{e}tale in codimension one cover $\tilde{X}\to X$ and a decomposition
        $$
            \tilde{X}\cong T\times\prod_{i}Y_i\times\prod_{j}Z_j
        $$
        where $T$ is a complex torus, $Y_i$ are irreducible Calabi-Yau orbifolds and $Z_j$ are irreducible symplectic orbifolds.
        (Here, a compact K\"{a}hler orbifold $Z$ is said to be irreducible Calabi-Yau when $K_Z\sim0$, $H^0(Z_\mathrm{reg},\Omega_{\regp{Z}}^p)=0$ for every $0<p<\mathrm{dim}(Z)$, and $Z_\mathrm{reg}$ is simply connected.) 
    \end{thm}

    An irreducible symplectic orbifold is always a primitively symplectic orbifold by the definition.
    However, the converse is not always true. The following example is such a case in two dimension.
    \begin{ex}\label{singkum}
        Let $T=\mathbb{C}^2/\Lambda$ be a 2-dimensional complex torus (where $\Lambda\subset\mathbb{C}^2$ is a lattice) and $\tau$ an involution on $T$ defined by $\tau(p)=-p$.
        Then the variety $Y:=T/\langle\tau\rangle$ (a singular Kummer surface) is a primitively simplectic orbifold and has 16 singularities.

        Let $\Sigma$ be an arbitrary subset of $\mathrm{Sing}(Y)$ and $X$  the blowup of $Y$ only at the points in $\Sigma$.
        Then $X$ is also a primitively simplectic orbifold.
        Let us think of $\pi_1(\regp{X})$.
        If $\Sigma=\emptyset$, we have $\pi_{1}(\regp{X})\cong\Lambda\rtimes\langle\tau\rangle$ and it is an infinite group.
        If $\Sigma=\mathrm{Sing}(Y)$, then $X$ is a Kummer surface and $\pi_{1}(\regp{X})=\{1\}$.
        Hence, the larger the subset $\Sigma$ is, the smaller $\pi_{1}(\regp{X})$ is.
        Indeed, we will prove the following claim in Example \ref{ex31}.
        \begin{clm}
            Assume that $\Sigma$ is nonempty.
            Let $p\in \Sigma$ be an arbitrary point. 
            Regard $\mathrm{Sing}(Y)$ as a group isomorphic to $(\Zn{2})^4$ whose identity is $p$ and let $\langle\Sigma\rangle$ be the subgroup of it generated by $\Sigma$.
            Then
            \[
                \pi_1(\regp{X})\cong (\Zn{2})^4/\langle\Sigma\rangle.
            \]
        \end{clm}


    \end{ex}
    

    \subsection*{Fujiki's examples}
    A.Fujiki constructed in \cite{Fu82} compact symplectic orbifolds of dimension four.
    We describe the construction of them briefly.
    More details are explained in Section 2.
    Let $S$ be a K3 surface or a 2-dimensional complex torus and $H$ a finite abelian subgroup of $\mathrm{Aut}(S)$ whose elements preserve the symplectic form on $S$.
    The group $H$ acts on $S\times S$ via the action defined by $h\cdot(x,y)=(hx,h^{-1}y)$.
    Let $\iota$ be the involution on $S\times S$ defined by $\iota(x,y)=(y,x)$.
    Let $G$ be the subgroup of $\mathrm{Aut}(S\times S)$ generated by $H$ and $\iota$ and let $Y=(S\times S)/G$.
    Then $Y$ is a symplectic orbifold whose singular locus has codimension 2.
    Let $f:X\to Y$ be a $\mathbb{Q}$-factorial terminalization of $Y$. 
    Then $X$ is also symplectic by the definition of terminalization, but not necessary primitively simplectic.
    When $X$ is primitive, we call it a Fujiki's example.
    The deformation class of $X$ is determined only by the isomorphism class of $H$, and the possibilities are classified as shown in Table \ref{k3tab} for the case where $S$ is a K3 surface and Table \ref{totab} for the case where $S$ is a complex torus (\cite[Theorem 13.1]{Fu82}). 
    \begin{table}[ht]
        \centering
        \caption{$S$: K3 surface}
        \label{k3tab}
        \begin{tabular}{cl}
            \toprule
            $H$ & Singularities of $X$ \\
            \midrule
            $\{1\}$ & smooth \\
            $\Zn{2}$ & $a_2=28$ \\
            $(\Zn{2})^{2}$ & $a_2=36$ \\
            $(\Zn{2})^{3}$ & $a_2=28$ \\
            $\Zn{3}$ & $a_3=15$ \\
            $(\Zn{3})^{2}$ & $a_3=12$ \\
            $\Zn{4}$ & $a_2=10,\; a_4=6$ \\
            $(\Zn{4})^{2}$ & $a_2=6$ \\
            $\Zn{6}$ & $a_2=9,\; a_3=6,\; a_6=1$ \\
            $\Zn{2}\times\Zn{4}$ & $a_2=12,\; a_4=4$ \\
            $\Zn{2}\times\Zn{6}$ & $a_2=12,\; a_3=3$ \\
            \bottomrule
        \end{tabular}
    \end{table}
    \begin{table}[ht]
        \centering
        \caption{$S$: complex torus}
        \label{totab}
        \begin{tabular}{cl}
            \toprule
            $H$ & Singularities of $X$ \\
            \midrule
            $\Zn{3}$ & $a_3=36$ \\
            $(\Zn{3})^2$ & $a_3=27$ \\
            $(\Zn{3})^3$ & smooth \\
            $\Zn{4}$ & $a_2=54,\; a_4=6$ \\
            $\Zn{2}\times\Zn{4}$ & $a_2=48,\; a_4=4$ \\ 
            $(\Zn{2})^2\times\Zn{4}$ & $a_2=36$ \\
            $\Zn{6}$ & $a_2=36,\; a_3=16$ \\
            \bottomrule
        \end{tabular}
    \end{table}
    Here, $a_k$ is the number of singular points of type $(\mathbb{C}^4/g_k, 0)$, where $g_k=\mathrm{diag}(\zeta_k,\zeta_k,\zeta_k^{-1}, \zeta_k^{-1})$ and $\zeta_k:=e^\frac{2\pi i}{k}$.

    \begin{rem}
        As pointed out in \cite[Remark 5.8]{menet}, there are some mistakes in \cite[Theorem 13.1]{Fu82} about the number of singular points of Fujiki's examples obtained from a K3 surface.
        We also found mistakes for Fujiki's examples obtained from a complex torus, and corrected them in Table \ref{totab}.
    \end{rem}
    
    As indicated in the tables, $X$ is smooth if $S$ is a K3 surface and $H\cong \{1\}$, or $S$ is a complex torus and $H\cong (\Zn{3})^{2}$.
    In the former case, $X$ is nothing but $\mathrm{Hilb}^2(S)$, 
    and in the latter case, $X$ is deformation equivalent to $\mathrm{Kum}^2(S)$ (cf. \cite[Proposition 14.3]{Fu82}).

    In \cite{Pe20}, A.Perego posed the following problem:
    \begin{prb}
        Which Fujiki's example is an irreducible symplectic orbifold? \label{prb}
    \end{prb}
    The purpose of this paper is to calculate $\pi_1(X_\mathrm{reg})$ for each Fujiki's example $X$ and give an answer for Problem \ref{prb}.
    The main result is the following:
    \begin{thm}\leavevmode
        \label{main}
        \begin{enumerate}[label=\textup{(\roman*)}]
            \item If $X$ is a Fujiki's example obtained from a K3 surface, then $\pi_1(X_\mathrm{reg})=1$, i.e., $X$ is an irreducible symplectic orbifold. \label{maink3}
            \item If $X$ is a Fujiki's example obtained from a 2-dimensional complex torus, then $\pi_1(X_\mathrm{reg})$ is as in Table \ref{pi1tab}:
            \begin{table}[htbp] 
                \caption{}
                \label{pi1tab}
                \centering
                    \begin{tabular}{cc} 
                        \toprule
                        $H$ & $\pi_1(X_\mathrm{reg})$ \\
                        \midrule
                        $\Zn{3}$ & $(\Zn{3})^2$ \\
                        $(\Zn{3})^2$ & $\Zn{3}$ \\
                        $(\Zn{3})^3$ & $\{1\}$ \\
                        $\Zn{4}$ & $(\Zn{2})^2$ \\
                        $\Zn{2}\times\Zn{4}$ & $\Zn{2}$ \\ 
                        $(\Zn{2})^2\times\Zn{4}$ & $\{1\}$ \\
                        $\Zn{6}$ & $\{1\}$ \\
                        \bottomrule
                    \end{tabular}
            \end{table}
            
        \end{enumerate}
    \end{thm}

    Consequently, four Fujiki's examples in Table \ref{pi1tab} are not irreducible symplectic.
    In these cases, the universal cover of $X_\mathrm{reg}$ extends to a finite cover $\tilde{X}$, which is \'{e}tale in codimension one cover. Then $\tilde{X}$ is an irreducible symplectic orbifold with only isolated singularities and is isomorphic to one of the remaining three Fujiki's examples in Table \ref{pi1tab} (cf. Remark \ref{rem}). 

    We briefly describe our method to prove Theorem \ref{main}.
    Firstly, we prove a general result.
    Let $Z$ be a simply connected complex manifold and $\mathcal{G}$ a discrete subgroup of $\mathrm{Aut}(Z)$ such that $Z/\mathcal{G}$ is an orbifold.
    Then for an arbitrary open subset $U\subset Z/\mathcal{G}$ containing the regular part of $Z/\mathcal{G}$, we define a certain normal subgroup $N_U\subset \mathcal{G}$ and prove that
    $\pi_1(U)\cong \mathcal{G}/N_U$ (Theorem \ref{pi1}).
    Let us return to the setting of Theorem \ref{main}.
    By Proposition \ref{Uprop}, we can take a Zariski open subset $U$ of $Y$ in such a way that
    \begin{enumerate}[label=\textup{(\roman*)}]
        \item $U\supset \regp{Y}$,
        \item $\regp{X}\supset f^{-1}(U)$,
        \item $\mathrm{Codim}(\regp{X}\setminus f^{-1}(U), \regp{X})\geq 2$. \label{cond3}
    \end{enumerate}
    Since $U$ has only quotient singularities, we have 
    \[
        f_*:\pi_1(f^{-1}(U))\stackrel{\sim}{\longrightarrow} \pi_1(U).
    \]
    On the other hand, we have an isomorphism
    \[
        \pi_1(f^{-1}(U))\stackrel{\sim}{\longrightarrow} \pi_1(\regp{X})
    \]
    by \ref{cond3}.
    Finally, by writing $Y$ in the form $Z/\mathcal{G}$ , we get
    \[
        \pi_1(\regp{X})\cong\pi_1(U)\cong\mathcal{G}/N_U
    \]
    and then we calculate $\mathcal{G}/N_U$.
    If $S$ is a K3 surface, we simply set $Z=S\times S$ and $\mathcal{G}=G$ (Corollary \ref{k3}).
    If $S=\mathbb{C}^2/\Lambda$ is a complex torus (where $\Lambda$ is a lattice), we set $Z=\mathbb{C}^2\times\mathbb{C}^2$ and $\mathcal{G}$ a certain discrete group containing $\Lambda\times\Lambda$ as a subgroup of finite index. In this case, the calculation becomes more complicated (Corollary \ref{tor}).

    As is well-known, $\mathrm{Hilb}^2(S)$ (where $S$ is a K3 surface) or $\mathrm{Kum}^2(S)$ (where $S$ is a 2-dimensional complex torus) is simply connected (cf. \cite[Lemma 6.1 and Proposition 7.8]{Be83}).
    Our method also gives a new proof to this fact.

    In \cite{menet}, G.Menet investigates the generalization of Fujiki's examples which are obtained from K3 surfaces, and have been proved results similar to Theorem \ref{pi1} and Theorem \ref{main} \ref{maink3} (\cite[Proposition 2.13]{menet}).

    After writing this article, we were informed by Mirko Mauri that they also obtained similar results in \cite[Proposition 1.14]{Mau}.

    
    \subsection*{Acknowledgements}
    The author would like to thank his advisor, Yoshinori Namikawa, for the helpful discussions, suggestions, and encouragement provided during this research.
    He also thanks G.Menet for helpful comments and
    discussions.
    He is grateful to his family for their understanding and constant support.

    \section{Fujiki's examples}
    \label{sec_fujiki}
    In this section, we describe the construction of Fujiki's examples in further detail.
    Let $S$ be either a K3 surface or a 2-dimensional complex torus and $H$ a finite abelian subgroup of $\mathrm{Aut}(S)$ such that each element in $H$ preserves the symplectic form on $S$. 
    If $S$ is a complex torus, we also assume that $H$ is not contained in the group of translations. 
    Then the possibilities of the isomorphism class of $H$ as a group are classified as follows (see {\cite[Prop. 12.5]{Fu82}}):

    \subsubsection*{(1) The case $S$ is a K3 surface.}
    \begin{enumerate}[label=(1.\alph*)]
        \item $(\Zn{2})^m, \quad(0\leq m\leq 4)$ \label{k1},
        \item $(\Zn{3})^m, \quad(1\leq m\leq 2)$ \label{k2},
        \item $(\Zn{4})^m, \quad(1\leq m\leq 2)$ \label{k33},
        \item $\Zn{p}, \quad(p=5,7,8)$ \label{k4},
        \item $\Zn{6}, $ \label{k4.5},
        \item $\Zn{2}\times\Zn{q}, \quad(q=4,6)$ \label{k5}. 
    \end{enumerate}
    \subsubsection*{(2) The case $S$ is a complex torus.}
    \begin{enumerate}[label=(2.\alph*)]
        \item $(\Zn{2})^m\times\Zn{2}, \quad(0\leq m\leq 4)$ \label{t2},
        \item $(\Zn{3})^m\times\Zn{3}, \quad(0\leq m\leq 2)$ \label{t3},
        \item $(\Zn{2})^m\times\Zn{4}, \quad(0\leq m\leq 2)$ \label{t4},
        \item $\Zn{6}\;(\cong\{1\}\times\Zn{6}) $ \label{t6}.
    \end{enumerate}

    In the torus cases, the first factor indicates the translation part, and the second factor indicates the linear part. 
    If we regard the translation part as a subgroup of $S$, all of its elements are fixed by the action of the linear part because of the abelianness.

    Define the action of $H$ on $S\times S$ by 
    $$
    \begin{array}{ll}
        h(x,y)=(hx,h^{-1}y) & \textup{($h\in H$, $(x,y)\in S\times S$ )}
    \end{array}
    $$
    and the involution $\iota\in \mathrm{Aut}(S\times S)$ by
    $$
    \begin{array}{ll}
        \iota(x,y)=(y,x) & \textup{($(x,y)\in S\times S$ )}.
    \end{array}
    $$
    Let $G$ be the subgroup of $\mathrm{Aut}(S\times S)$ generated by $H$ and $\iota$.
    Then for any $h\in H$ and $(x,y)\in S\times S$,
    \begin{align*}
        \iota h\iota^{-1}(x,y)
        &=\iota h\iota(x,y) \\
        &=\iota h(y,x) \\
        &=\iota (hy,h^{-1}x) \\
        &=(h^{-1}x,hy) \\
        &=h^{-1}(x,y).
    \end{align*}
    Hence, $\iota h\iota^{-1}=h^{-1}\in H$, and so $G=H\rtimes \langle\iota\rangle$.

    Let $Y=(S\times S)/G$ and $f:X\to Y$ a $\mathbb{Q}$-factorial terminalization of $Y$. 
    This exists by \cite[Corollary 1.4.3]{BCHM}. 
    Except in the case of \ref{k4}, $X$ is constructed explicitly in \cite[\S 7]{Fu82} and it is proven that $X$ has only quotient singularities.
    Therefore, $X$ is a symplectic orbifold except in the case of \ref{k4}.
    Moreover, \cite[Lemma 13.4]{Fu82} shows that $X$ is a primitively symplectic except the case \ref{t2}.
    Consequently, except in the cases of \ref{t2} and \ref{k4},
    we obtain a primitively symplectic orbifold $X$, which we call a Fujiki's example.

    Let us observe the singularities of $Y$, which are the images of fixed points on $S\times S$ by the action of $G$.
    Since $G=H\rtimes \langle\iota\rangle$, each element $g\in G\setminus\{1\}$ is expressed in the form either $h$ or $\iota h$ for some $h\in H$.
    In the case $g=\iota h$ we have
    \begin{align*}
        \mathrm{Fix}_{S\times S}(\iota h)
        &=\{(x,y)\in S\times S\mid (h^{-1}y,hx)=(x,y)\}\\
        &=\{(x,y)\in S\times S\mid y=hx\}\\
        &=\{(x,hx)\in S\times S\mid x\in S\}.
    \end{align*}
    Thus, $\mathrm{Fix}_{S\times S}(\iota h)\,(\cong S$) is a 2-dimensional locus on $S\times S$ and its image by the natural projection $\pi:S\times S\to Y$ is a 2-dimensional singlular locus of $Y$.
    In the case $g=h$ we have
    \begin{align*}
        \mathrm{Fix}_{S\times S}(h)
        &=\{(x,y)\in S\times S\mid (hx,h^{-1}y)=(x,y)\}\\
        &=\{(x,y)\in S\times S\mid hx=x, hy=y\}\\
        &=\{(x,y)\in S\times S\mid x,y\in \mathrm{Fix}_S(h)\}.
    \end{align*}
    Since $h\,(\ne 1)$ is a symplectic automorphism on $S$, it has only isolated fixed points.
    This implies that $\mathrm{Fix}_{S\times S}(h)$ consists of finitely many points.
    Let $p$ be a point of $\mathrm{Fix}_{S\times S}(h)$.
    There are two possibilities of the image $\pi(p)$.
    The first one is an intersection point of 2-dimensional singular loci and the second one is a isolated singular point.
    In summary, $Y$ has 2-dimensional singular loci (which intersect each other) and isolated singlular points.

    For later use, we prove the next proposition.
    \begin{prop}\label{Uprop}
        Define an open subset $U$ of $Y$ by
        \[
            U=Y\setminus \pi\left(\bigcup_{h\in H\setminus\{1\}}\mathrm{Fix}_{S\times S}(h)\right).
        \]
        Then $U$ satisfies the following conditions:
        \begin{enumerate}[label=\textup{(\roman*)}]
            \item $U\supset \regp{Y}$, \label{cond21}
            \item $\regp{X}\supset f^{-1}(U)$, \label{cond22}
            \item $\mathrm{Codim}(\regp{X}\setminus f^{-1}(U), \regp{X})\geq 2$. \label{cond23}
        \end{enumerate}

    \end{prop}
    \begin{proof}
        The open subset $U$ is obtained by substructing isolated singular points and intersection points of 2-dimensional singular loci from $Y$.
        Thus \ref{cond21} is obvious.

        Because
        $$
            f|_{f^{-1}(U)}:f^{-1}(U)\to U
        $$
        is a resolution of quotient singularities, \ref{cond22} also holds.

        By \cite[Corollary 0.2]{Na22}, $\mathrm{dim}(f^{-1}(y))\leq 2$ for any point $y\in Y$.
        Since $Y\setminus U$ consists of finitely many points,
        we have
        \begin{align*}
            \mathrm{Codim}(\regp{X}\setminus f^{-1}(U), \regp{X})
            &=\mathrm{Codim}(\regp{X}\cap f^{-1}(Y\setminus U), \regp{X})\\
            &\geq 2.
        \end{align*}
    \end{proof}

    \begin{cor}\label{Ucor}
        Let $U$ be the open subset of $Y$ defined in Proposition \ref{Uprop}.
        Then
        \[
            \pi_1(\regp{X})\cong\pi_1(U).
        \]
    \end{cor}
    \begin{proof}
        By {\cite[Thm. 7.8]{Kol93}}, we have that
        \[
            f_*:\pi_1(f^{-1}(U))\stackrel{\sim}{\longrightarrow} \pi_1(U).
        \]
        Also, we see
        \[
            \pi_1(f^{-1}(U))\stackrel{\sim}{\longrightarrow} \pi_1(\regp{X})
        \]
        by \ref{cond3}.
        Therefore, the conclusion holds.
    \end{proof}

    \section{Calculation of $\pi_1(X_\mathrm{reg})$}
    \label{sec_pi1}
    In this section, we calculate $\pi_1(X_\mathrm{reg})$ for Fujiki's examples $X$.
    First, we prove the following (general) theorem.
    \begin{thm}
        Let $Z$ be a simply connected complex manifold.
        Let $\mathcal{G}$ be a discrete subgroup of $\mathrm{Aut}(Z)$ which satisfies the following conditions:
        \begin{enumerate}[label=\textup{(\roman*)}]
            \item the action of $\mathcal{G}$ on $Z$ is properly discontiuous, \label{pi1cond1}
            \item the stabilizer group $\mathrm{Stab}(z)$ is finite for all $z\in Z$, \label{pi1cond2}
            \item $\mathrm{Codim}(\mathrm{Fix}(g), Z)\geq 2$ for all $g\in \mathcal{G}\setminus \{1\}$.
        \end{enumerate}
        Define $Y=Z/\mathcal{G}$ and let $\pi:Z\to Y$ be the natural surjection
        (because of the conditions \ref{pi1cond1} and \ref{pi1cond2}, $Y$ is a complex orbifold).
        Finally, let $U$ be an arbitrary open subset of $Y$ containing $Y_\mathrm{reg}$ and define $N_U$ to be the subgroup of $\mathcal{G}$ generated by the set $$\bigcup_{z\in\pi^{-1}(U)}\mathrm{Stab}(z).$$
        Then $N_U$ is a normal subgroup of $\mathcal{G}$ and $\pi_1(U)\cong \mathcal{G}/N_U$. 
        \label{pi1}
    \end{thm}
    \begin{proof}
        For all $g\in \mathcal{G}$, $z\in \pi^{-1}(U)$ and $h\in \mathrm{Stab}(z)$, $hz\in\pi^{-1}(U)$ and $ghg^{-1}\in \mathrm{Stab}(hz)$.
        Hence, $N_U$ is a normal subgroup of $\mathcal{G}$.

        Let $p$ be the projection $Z/N_U\to Y$.
        Then $p^{-1}(U)\to U$ is unramified.
        Indeed, let $\bar{z}$ be an arbitrary point in $p^{-1}(U)$ and take a lift $z\in \pi^{-1}(U)$ of $\bar{z}$.
        The covering $p:p^{-1}(U)\to U$ is a Galois covering whose Galois group equals to $\mathcal{G}/N_U$. The stabilizer group of $\mathrm{Stab}(\bar{z})\subset \mathcal{G}/N_U$ is trivial because
        $$\mathrm{Stab}(\bar{z})
        =\mathrm{Stab}(z)N_U/N_U
        =\{1\}$$ and this implies $p:p^{-1}(U)\to U$ is unramified.

        Moreover, $p^{-1}(U)$ is simply connected.
        To see this, let $q:V\to p^{-1}(U)$ be the universal covering of $p^{-1}(U)$ and assume that $q$ is not the trivial covering.
        Since the composition map $$p\circ q:V\to p^{-1}(U)\to U$$ is also the universal covering of $U$, its restriction $$p\circ q:(p\circ q)^{-1}(Y_\mathrm{reg})\to p^{-1}(Y_\mathrm{reg})\to Y_\mathrm{reg}$$ is an unramified Galois covering.
        On the other hand, since $Z$ is simply connected complex manifold and $\pi^{-1}(Y_\mathrm{reg}$) is the open subset of $Z$ such that the (complex) codimension of its complement is at least two, $\pi^{-1}(Y_\mathrm{reg})$ is also simply connected and $\pi^{-1}(Y_\mathrm{reg})$ is the universal covering of $Y_\mathrm{reg}$.
        Therefore, $\pi:\pi^{-1}(Y_\mathrm{reg})\to Y_\mathrm{reg}$ factors as 
        \begin{equation}
            \pi^{-1}(Y_\mathrm{reg})\to (p\circ q)^{-1}(Y_\mathrm{reg})\xrightarrow{q} p^{-1}(Y_\mathrm{reg})\xrightarrow{p} Y_\mathrm{reg}.
            \label{seq}
        \end{equation}
        Let $N'$ be the normal subgroup of $\mathcal{G}$ which corresponds to $(p\circ q)^{-1}(Y_\mathrm{reg})$. 
        The sequence \eqref{seq} corresponds to the sequence of groups
        $$
            1\subset N' \subset N_U \subset \mathcal{G}.
        $$
        Since $q$ is not trivial covering, $N'\subsetneq N_U$ and hence, there exists a $z\in \pi^{-1}(U)$ such that $\mathrm{Stab}(z)\not\subset N'$.
        If $z'$ is the image of $z$ in $\pi^{-1}(U)/N'=V$, then 
        $$
            \mathrm{Stab}(z')(\subset \mathcal{G}/N')=\mathrm{Stab}(z)N'/N'\neq\{1\}.
        $$
        This contradicts the fact that $p\circ q:V\to p^{-1}(U)\to U$ is unramified.
        Therefore, $q$ is the trivial covering and $p^{-1}(U)$ is simply connected.

        Consequently, $p: p^{-1}(U)\to U$ is the universal covering.
        From this fact and the definition of $p$, we deduce that
        $$
            \pi_1(U)\cong \mathrm{Aut}(p^{-1}(U)/U)\cong \mathcal{G}/N_U.
        $$
    \end{proof}

    \begin{ex}
        As an example of the theorem, we calculate the fundamental group of the regular part of the primitively simplectic orbifold obtained by blowing up some of the singularities of a singular Kummer surface.

        Let $\Lambda\subset\mathbb{C}^2$ be a 2-dimensional lattice and $T=\mathbb{C}^2/\Lambda$ a 2-dimensional complex torus.
        Let $\tau$ be an involution $T$ is defined by $\tau(p)=-p$ and let $Y=T/\langle\tau\rangle$ be the singular Kummer surface associated with the complex torus $T$.
        Let $\mathrm{Sing}(Y)=\{p_1,\dots,p_{16}\}$ be the set of singular points of $Y$.
        Fix a nonempty subset $\Sigma\subset\mathrm{Sing}(Y)$ and let $X$ be a surface obtained by blowing up $X$ at the points in $\Sigma$.
        We prove the following claim.
        \begin{clm}
            Let $p\in \Sigma$ be an arbitrary point. 
            Regard $\mathrm{Sing}(Y)=\{p_1,\dots,p_{16}\}$ as a group isomorphic to $(\Zn{2})^4$ whose identity is $p$ and let $\langle\Sigma\rangle$ be the subgroup of it generated by $\Sigma$.
            Then
            \[
                \pi_1(\regp{X})\cong (\Zn{2})^4/\langle\Sigma\rangle.
            \]
        \end{clm}

        \begin{proof}
        Let $f:X\to Y$ be the blowup morphism and 
        let $U$ an open subset of $Y$ obtained by getting rid of singular points not belonging to the set $\Sigma$:
        \begin{align*}
            U:=&Y\setminus(\mathrm{Sing}(Y)\setminus \Sigma)\\
            =&Y_\mathrm{reg}\sqcup \Sigma. 
        \end{align*}
        Then the restriction map
        $$
            f\vert_{X_\mathrm{reg}}: X_\mathrm{reg}\to U 
        $$
        is a resolution of quotient singularities.
        Therefore, by {\cite[Thm. 7.8]{Kol93}},
        $$
            \pi_1(X_\mathrm{reg})\stackrel{\sim}{\longrightarrow}\pi_1(U).
        $$
        Hence, we compute $\pi_1(U)$ by applying Theorem \ref{pi1}.
        In order to do that, we write $Y$ as a quotient of a simply connected complex manifold: 
        \begin{align*}
            Y&=T/\langle\tau\rangle \\
            &=(\mathbb{C}^2/\Lambda)/\langle\tau\rangle \\
            &=\mathbb{C}^2/(\Lambda\rtimes\langle\tau\rangle),
        \end{align*}
        where the action of $\tau$ on $\mathbb{C}^2$ is similarly defined by $\tau(x)=-x$ for any $x\in \mathbb{C}^2$.
        We apply Theorem \ref{pi1} by setting $Z=\mathbb{C}^2$ and $\mathcal{G}=\Lambda\rtimes\langle\tau\rangle$.
        Then we get
        \begin{equation*}
            \pi_1(U)
            \cong \mathcal{G}/N_U. 
        \end{equation*}
        and we determine the normal subgroup $N_U$ in the next lemma.
        \begin{lem*}
            Let $\pi: \mathbb{C}^2\to Y(=\mathbb{C}^2/\mathcal{G})$ be the natural projection and
            $\Lambda'$ be the sublattice of $\Lambda$ generated by the set $$\left\{\lambda\in \Lambda\mid\pi\left(\frac{\lambda}{2}\right)\in \Sigma\right\}.$$
            Then  
            $$N_U=\Lambda'\rtimes \langle\tau\rangle.$$
        \end{lem*}
        \begin{proof}[Proof of Lemma]
            For any $\lambda\in \Lambda$, let $t_\lambda\in \mathrm{Aut}(\mathbb{C}^2)$ be the translation map: $$
            t_\lambda:\mathbb{C}^2 \ni x \mapsto x+\lambda\in \mathbb{C}^2.
            $$
            Since $G=\Lambda\rtimes\langle\tau\rangle$, we can write each element of $G$ in the form either $t_\lambda$ or $t_\lambda\tau$.

            Let us recall the definition of the normal subgroup $N_U\triangleleft\; G$ in Theorem \ref{pi1}:
            $$
                N_U:=\left\langle\bigcup_{x\in\pi^{-1}(U)}\mathrm{Stab}(x)\right\rangle .
            $$
            If $t_\lambda$ belongs to $\mathrm{Stab}(x)$ for some $x\in\pi^{-1}(U)$, then
            \begin{align*}
                &t_\lambda (x)=x \\
                \Leftrightarrow& x+\lambda=x \\
                \Leftrightarrow& \lambda=0.
            \end{align*}
            Hence, any non-trivial translation element $t_\lambda$ cannot be in $\mathrm{Stab}(x)$.
            On the other hand, if $t_\lambda\tau\in\mathrm{Stab}(x)$ for some $x\in\pi^{-1}(U)$, then
            \begin{align*}
                &t_\lambda\tau (x)=x \\
                \Leftrightarrow& -x+\lambda=x \\
                \Leftrightarrow& x=\frac{\lambda}{2}.
            \end{align*}
            Therefore, the image of $x$ on $T$ is a 2-torsion point, and so $\pi(x)$ is a singlular point of $Y$.
            Because $\pi(x)\in U=Y_\mathrm{reg}\sqcup \Sigma$, we deduce that $\pi\left(\frac{\lambda}{2}\right)=\pi(x)\in \Sigma$.
            Hence, for any $\lambda\in\Lambda$ such that $\pi\left(\frac{\lambda}{2}\right)\in \Sigma$, we have
            \begin{align*}
                t_{\lambda}\tau\left(\frac{\lambda}{2}\right)
                =-\frac{\lambda}{2}+\lambda
                =\frac{\lambda}{2}
            \end{align*}
            and therefore
            \[
                t_{\lambda}\tau\in \mathrm{Stab}\left(\frac{\lambda}{2}\right).
            \]
            Consequently, we get
            \begin{align*}
                N_U
                &=\left\langle\bigcup_{x\in\pi^{-1}(U)}\mathrm{Stab}(x)\right\rangle\\
                &=\left\langle\left\{t_{\lambda}\tau\mid\pi\left(\frac{\lambda}{2}\right)\in \Sigma\right\}\right\rangle\\
                &=\left\langle\left\{t_{\lambda}\mid\pi\left(\frac{\lambda}{2}\right)\in \Sigma\right\}\right\rangle\rtimes\langle\tau\rangle\\
                &=\Lambda'\rtimes\langle\tau\rangle.
            \end{align*}
        \end{proof}
        We return to the proof of the claim.
        From the above, we have
        \begin{align*}
            \pi_1(X_\mathrm{reg})
            &\cong\pi_1(U)\\
            &\cong G/N_U\\
            &\cong (\Lambda\rtimes\langle\tau\rangle)/(\Lambda'\rtimes\langle\tau\rangle)\\
            &\cong \Lambda/\Lambda'\\
            &\cong (\Lambda/2\Lambda)/(\Lambda'/2\Lambda)\\
            &\cong (\Zn{2})^4/\langle\Sigma\rangle.
        \end{align*}
        \label{ex31}
        \end{proof}
    \end{ex}

    By means of Theorem \ref{pi1}, we calculate $\pi_1(X_\mathrm{reg})$ for Fujiki's examples $X$.
    Let $S$ be either a K3 surface or a complex torus and $H$ the finite abelian subgroup of $\mathrm{Aut}(S\times S)$ which is used to construct $X$.
    Let $G:=\langle H,\iota \rangle \subset \mathrm{Aut}(S\times S)$ be the subgroup of $\mathrm{Aut}(S\times S)$ defined in Section \ref{sec_fujiki}, where $\iota\in \mathrm{Aut}(S\times S)$ is the interchanging involution.
    Let $Y=(S\times S)/G$.
    Then a Fujiki's example $X$ is obtained as a $\mathbb{Q}$-factorial terminalization $f:X\to Y$.

    We firstly deal with the cases where $S$ is a K3 surface.

    \begin{cor}
        Let $X$ be a Fujiki's example obtained from a K3 surface $S$.
        Then $\pi_1(X_\mathrm{reg})=\{1\}$.
        Thus, $X$ is an irreducible symplectic orbifold.
        \label{k3}
    \end{cor}
    \begin{proof}
        Define an open subset $U$ of $Y$ as in the Proposition \ref{Uprop}.
        Then by Corollary \ref{Ucor}, we have $\pi_1(X_\mathrm{reg})\cong \pi_1(U)$.

        Therefore, it is enough to show that $\pi_1(U)=\{1\}$.
        Let us apply Theorem \ref{pi1} by setting $Z=S\times S$ and $\mathcal{G}=G$.
        Then $\pi(U)\cong G/N_U$ holds.
        We need to show that $G=N_U$

        Let $h\in H$ be an arbitrary element. 
        Since $$\mathrm{dim}\,\mathrm{Fix}_{S\times S}(\iota h)=\mathrm{dim}\,\pi(\mathrm{Fix}_{S\times S}(\iota h))=2$$ and $$\mathrm{dim}\,Y\setminus U=\mathrm{dim}\,\pi\left(\bigcup_{h\in H\setminus\{1\}}\mathrm{Fix}_{S\times S}(h)\right)=0,$$ we can choose a point $p\in\mathrm{Fix}_{S\times S}(\iota h)$ so that $\pi(p)\in U$.
        Hence, 
        $$
            \iota h\in\mathrm{Stab}(p)\subset N_U.
        $$
        In particular, we have $\iota\in N_U$ by setting $h=1_H$.
        Therefore, for any $h\in H$, we see that 
        $$
            h=\iota\cdot\iota h\in N_U.
        $$
        As a result, 
        $$
            G=H\rtimes \langle\iota\rangle\subset N_U.
        $$

    \end{proof}

    Next, we calculate for the cases where $X$ comes from a complex torus.
    The same argument as above cannot be applied to these cases because $S\times S$ is not simply connected when $S$ is a 2-dimensional complex torus.
    \begin{cor}
        \label{tor}
        Let $X$ be a Fujiki's example obtained from a 2-dimensional complex torus $S$.
        Then
        $$
        \pi_1(X_\mathrm{reg})\cong\left\{
        \begin{array}{ll}
            (\Zn{3})^{2-m} & \textup{(for  the case \ref{t3})}\\
            (\Zn{2})^{2-m} & \textup{(for  the case \ref{t4})}\\
            \{1\} & \textup{(for  the case \ref{t6})}
        \end{array}
        \right. .
        $$

    \end{cor}
    \begin{proof}
        Let us write $S=\mathbb{C}^2/\Lambda$ where $\Lambda\subset\mathbb{C}^2$ is a lattice.
        We choose the image of the origin of $\mathbb{C}^2$ as a origin of $S$.
        Let $U$ an open subset of $Y$ defined in the Proposition \ref{Uprop}.
        Then by Corollary \ref{Ucor}, we have $\pi_1(X_\mathrm{reg})\cong \pi_1(U)$.
        We calculate $\pi_1(U)$.

        Since $H(\subset\mathrm{Aut}(S))$ is abelian, we have the decomposition $H= H'\times H_0$ where $H'$ is the translation part and $H_0$ is the linear part of $H$.
        We can regard $H'$ as a finite subgroup of $S$ whose elements are preserved by the action of $H_0$.
        Then
        \begin{align*}
            G
            &\cong H\rtimes \langle\iota\rangle\\
            &\cong (H'\times H_0)\rtimes \langle\iota\rangle\\
            &\cong H'\rtimes (H_0\rtimes\langle\iota\rangle).\\
        \end{align*}
        This is the decomposition of $G$ as a subgroup of $\mathrm{Aut}(S\times S)$; $H'$ is the translation part and $H_0\rtimes\langle\iota\rangle$ is the linear part. 
        The group $H'$ is regarded as a subgroup of the 4-dimensional torus $S\times S$ via the inclusion 
        \begin{align*}
            i: H'\ni\gamma\mapsto(\gamma,-\gamma)\in S\times S.
            \label{emb}
        \end{align*}

        Let $n$ be the order of $H_0$ (by the classification of $H$, $n\in\{3,4,6\}$ and $H_0\cong\Zn{n}$).
        Then $H_0\rtimes\langle\iota\rangle$ is isomorphic to the dihedral group $D_n$ of order $2n$.
        Therefore, $G\cong H'\rtimes D_n$.
        We have 
        \begin{align*}
            (S\times S)/G
            &\cong(\mathbb{C}^4/\Lambda^2)/G\\
            &\cong\mathbb{C}^4/(\Lambda^2\rtimes G)\\
            &\cong\mathbb{C}^4/(\Lambda^2\rtimes (H'\rtimes D_n))\\
            &\cong\mathbb{C}^4/((\Lambda^2\rtimes H')\rtimes D_n).
        \end{align*}
        Since the action of $H'$ is translation, $H'$ and $\Lambda^2$ commutes.
        Hence, $\Lambda^2\rtimes H'\cong\Lambda^2\times H'$.
        This group is an overlattice of $\Lambda^2$ corresponds to the quotient torus $(S\times S)/H'$.
        We denote this overlattice by $\Gamma$.
        Finally, we have
        $$
            Y=(S\times S)/G\cong\mathbb{C}^4/(\Gamma\rtimes D_n).
        $$

        We explain the action of $\Gamma\rtimes D_n$ on $\mathbb{C}^4$.
        The lattice $\Gamma$ acts as translations.
        Let $\sigma$ be a generator of $H_0\cong\Zn{n}$.
        The group $H_0$ originally acts on $S=\mathbb{C}^2/\Lambda$ as the linear part of $H$.
        Hence, $\sigma$ naturally lifts to a linear automorphism of $\mathbb{C}^2$ which preserves the lattice $\Lambda$.
        Because $H$ acts on $S\times S$ by $h(x,y)=(hx, h^{-1}y)$,
        the lifted linear action of $\sigma$ on $\mathbb{C}^2\times\mathbb{C}^2=\mathbb{C}^4$ is also expressed as 
        $$
            \sigma(x,y)=(\sigma x,\sigma^{-1} y),\quad(x,y)\in\mathbb{C}^2\times\mathbb{C}^2.
        $$
        Similarly, the involution $\iota\in \mathrm{Aut}(S\times S)$ lifts to $\mathrm{Aut}(\mathbb{C}^4)$ by the action
        $$
            \iota(x,y)=(y, x),\quad(x,y)\in\mathbb{C}^2\times\mathbb{C}^2=\mathbb{C}^4.
        $$
        These lifts $\sigma$ and $\iota$ generate $H_0\rtimes\langle\iota\rangle=D_n$,
        and preserve the lattice $\Gamma\subset\mathbb{C}^4$.



        Let $\pi:\mathbb{C}^4\to Y=\mathbb{C}^4/(\Gamma\rtimes D_n)$ be the natural projection.
        Then by applying Theorem \ref{pi1} (by setting $Z=\mathbb{C}^4$ and $\mathcal{G}=\Gamma\rtimes D_n$), we have 
        $$
            \pi_1(U)\cong(\Gamma\rtimes D_n)/N_U,
        $$
        where $N_U=\langle\bigcup_{p\in\pi^{-1}(U)}\mathrm{Stab}(p)\rangle$ is the normal subgroup of $\Gamma\rtimes D_n$ defined Theorem \ref{pi1}.
        Let us calculate $N_U$.
        For each $k=0,\ldots,n-1$, 
        let 
        $$
            P_k=\{(x,y)\in \mathbb{C}^2\times\mathbb{C}^2=\mathbb{C}^4\mid y=-\sigma x\}.
        $$
        Each $P_k$ is a 2-plane in $\mathbb{C}^4$.
        We set $\Gamma_k=\Gamma\cap P_k$ and 
        $$
            \Gamma'=\sum_{k=0}^{n-1}\Gamma_k\quad(\subset \Gamma) 
        $$
        We prove the following claim.
    \begin{clm}\leavevmode
        \begin{enumerate}[label=\textup{(\roman*)}]
            \item $N_U=\Gamma'\rtimes D_n$ as a subgroup of $\Gamma\rtimes D_n$,
            \item $\Gamma'=\Gamma_0\oplus\Gamma_1$,
            \item Define $\phi:\mathbb{C}^4\stackrel{\sim}{\longrightarrow}\mathbb{C}^4$ by
            $$
                \phi(x,y)=(x,\sigma y),\quad (x,y)\in\mathbb{C}^2\times\mathbb{C}^2=\mathbb{C}^4.
            $$
            Then $\Gamma_1=\phi(\Gamma_0)$
        \end{enumerate}
    \end{clm}
    \begin{proof}[Proof of Claim]\leavevmode
    \begin{enumerate}[label=\textup{(\roman*)}]
        \item
        For each $\gamma\in\Gamma$, let $t_\gamma\in \mathrm{Aut}(\mathbb{C}^4)$ be the translation $p\mapsto p+\gamma$.
        Then each element $g\in\Gamma\rtimes D_n$ is expressed in the form either $t_\gamma\sigma^k$ or $t_\gamma\iota\sigma^k$ for some $\gamma\in\Gamma$ and $k\in\{0,\ldots,n-1\}$.
        Let us suppose that  $$t_\gamma\sigma^k\in\bigcup_{p\in\pi^{-1}(U)}\mathrm{Stab}(p).$$
        Then for some $p\in\pi^{-1}(U)$, we have 
        \begin{align*}
            p
            &=t_\gamma\sigma^k(p)
            =\sigma^kp+\gamma.
        \end{align*}
        Therefore,
        $$
            \pi(p)=\sigma^k\pi(p).
        $$
        Since this contradicts the definition of $U$,
        any element in the form of $t_\gamma\sigma^k$ dose not belong to $\bigcup_{p\in\pi^{-1}(U)}\mathrm{Stab}(p)$.
        On the other hand, if 
        $$
            t_\gamma\iota\sigma^k\in\bigcup_{p\in\pi^{-1}(U)}\mathrm{Stab}(p),
        $$
        then $\mathrm{Fix}_{\mathbb{C}^4}(t_\gamma\iota\sigma^k)\ne\emptyset$.
        Conversely, if $\mathrm{Fix}_{\mathbb{C}^4}(t_\gamma\iota\sigma^k)\ne\emptyset$, then 
        $$
            \mathrm{Fix}_{\mathbb{C}^4}(t_\gamma\iota\sigma^k)\setminus
            \bigcup_{\gamma\in\Gamma, k=0,\ldots,n-1}\mathrm{Fix}_{\mathbb{C}^4}(t_\gamma\sigma^k)\ne\emptyset
        $$
        because the first term is 2-dimensional locus and the second term is 0-dimensional locus.
        Let $p$ be a point of the above set.
        Again by the definition of $U$, we have $\pi(p)\in U$
        and therefore, 
        $$
            t_\gamma\iota\sigma^k\in\bigcup_{p\in\pi^{-1}(U)}\mathrm{Stab}(p).
        $$
        Consequently,
        $$
            \bigcup_{p\in\pi^{-1}(U)}\mathrm{Stab}(p)=\{t_\gamma\iota\sigma^k\mid\mathrm{Fix}_{\mathbb{C}^4}(t_\gamma\iota\sigma^k)\ne\emptyset\}.
        $$
        Moreover, we claim that 
        $$
        \mathrm{Fix}_{\mathbb{C}^4}(t_\gamma\iota\sigma^k)\ne\emptyset\Leftrightarrow\gamma\in\Gamma_k.
        $$
        If $p\in\mathrm{Fix}_{\mathbb{C}^4}(t_\gamma\iota\sigma^k)$, 
        then 
        \begin{align*}
            \iota\sigma^kp+\gamma=p
            \Leftrightarrow\gamma=p-\iota\sigma^kp
        \end{align*}
        and hence, 
        \begin{align*}
            \iota\sigma^k(\gamma)
            &=\iota\sigma^k(p-\iota\sigma^kp)\\
            &=\iota\sigma^kp-p\\
            &=\gamma.
        \end{align*}
        This implies $\gamma\in P_k\cap\Gamma=\Gamma_k$. 
        On the other hand, if $\gamma\in\Gamma_k$, 
        then 
        \begin{align*}
            t_\gamma\iota\sigma^k\left(\frac{\gamma}{2}\right)
            &=t_\gamma\left(-\frac{\gamma}{2}\right)\\
            &=-\frac{\gamma}{2}+\gamma\\
            &=\frac{\gamma}{2}
        \end{align*}
        and so 
        $$
        \frac{\gamma}{2}\in\mathrm{Fix}_{\mathbb{C}^4}(t_\gamma\iota\sigma^k)\ne\emptyset.
        $$
        Finally, we have
        \begin{align*}
            N_U
            &=\langle\bigcup_{p\in\pi^{-1}(U)}\mathrm{Stab}(p)\rangle\\
            &=\langle\{t_\gamma\iota\sigma^k\mid\gamma\in\Gamma_k,k=0.\ldots,n-1\}\rangle\\
            &=\langle\{t_\gamma\mid\gamma\in\bigcup_{k=0}^{n-1}\Gamma_k\}\cup D_n\rangle\\ 
            &=\langle\{t_\gamma\mid\gamma\in\Gamma'\}\cup D_n\rangle\\ 
            &=\Gamma'\rtimes D_n \quad(\subset \Gamma\rtimes D_n). 
        \end{align*}

        \item 
        First, we show that $\Gamma_k\subset\Gamma_0+\Gamma_1$ for all $k=0,\ldots,n-1$.
        This implies $\Gamma'=\Gamma_0+\Gamma_1$.
        We prove by induction on $k$. 
        Let $\gamma=(x,y)\in \Gamma_k$.
        Then $y=-\sigma^k x$
        and for arbitrary $l$, we have
        \begin{align*}
            \sigma^l\gamma
            &=\sigma^l(x, -\sigma^k x)\\
            &=(\sigma^l x, -\sigma^{k-l}x)\\
            &=(\sigma^l x, -\sigma^{k-2l}(\sigma^l x))\\
            &\in \Gamma_{k-2l}.
        \end{align*}
        We also have
        \begin{align*}
            \gamma+\sigma^{l}\gamma
            &=(x+\sigma^l x, -\sigma^k x-\sigma^{k-l}x)\\
            &=(x+\sigma^l x, -\sigma^{k-l}(x+\sigma^l x))\\
            &\in \Gamma_{k-l}.
        \end{align*}
        In the case where $k$ is even, let $l=\frac{k}{2}$.
        Then $\sigma^{l}\gamma\in\Gamma_0$ and $\gamma+\sigma^{l}\gamma\in\Gamma_{l}$.
        By the induction hypothesis, $\Gamma_{l}\subset \Gamma_0+\Gamma_1$.
        Hence, 
        $$
            \gamma=-\sigma^{l}\gamma+(\gamma+\sigma^l\gamma)\in\Gamma_0+\Gamma_1
        $$
        Similarly, in the case where $k$ is odd, let $l=\frac{k-1}{2}$.
        Then $\sigma^{l}\gamma\in\Gamma_1$ and $\gamma+\sigma^{l}\gamma\in\Gamma_{l+1}\subset \Gamma_0+\Gamma_1$.
        Hence, $\gamma\in\Gamma_0+\Gamma_1$.
        Therefore, we have proved that $\Gamma'=\Gamma_0+\Gamma_1$.
        Since $\Gamma_0\cap\Gamma_1\subset P_0\cap P_1=\{0\}$,
        we have $\Gamma'=\Gamma_0\oplus\Gamma_1$.

        \item 
        Since $\Gamma$ is preserved by $\phi$, we have
        \begin{align*}
            \Gamma_1
            &=P_1\cap \Gamma\\
            &=\phi(P_0)\cap \phi(\Gamma)\\
            &=\phi(P_0\cap\Gamma)\\
            &=\phi(\Gamma).
        \end{align*}

    \end{enumerate}
    \end{proof}
        By the claim, we have
        \begin{align*}
            \pi_1(X_\mathrm{reg})
            &\cong \pi(U)\\
            &\cong (\Gamma\rtimes D_n)/(\Gamma'\rtimes D_n)\\
            &\cong \Gamma/\Gamma'\\
            &\cong \Gamma/(\Gamma_0\oplus\Gamma_1)\\
            &\cong \Gamma/(\Gamma_0\oplus\phi(\Gamma_0))
        \end{align*}
        We calculate $\Gamma_0$ and $\phi(\Gamma_0)$.
        Let us consider, for example, the case of \ref{t4} and $m=1$, i.e., the case where $H\cong\Zn{2}\times\Zn{4}$.
        The calculations for the other cases are similar.

        We can choose a basis $e_1,e_2,e_3,e_4$ for the lattice $\Lambda\subset\mathbb{C}^2$ so that 
        \begin{align*}
            &\sigma (e_1)=e_2,\quad \sigma (e_2)=-e_1,\quad \sigma (e_3)=e_4,\quad \sigma (e_4)=-e_3
        \end{align*}
        (see the proof of \cite[Theorem 12,.7]{Fu82}).
        Let $f=\frac{e_1+e_2}{2}\in\mathbb{C}^2$.
        Then 
        \begin{align*}
            \sigma(f)
            &=\frac{-e_2+e_1}{2}\\
            &=f-e_2\\
            &\in f + \Lambda
        \end{align*}
        and so the image of $f$ in $S=\mathbb{C}^2/\Lambda$ is fixed by $\sigma$.
        Without loss of generality, we may assume that $H'=\Zn{2}$ (the translation part of $H$) is generated by this image. 
        Since $H'$ is embedded into $S^2$ by the inclusion $i$ and $\Gamma\subset\mathbb{C}^4$ is the overlattice of $\Lambda^2$ corresponds to this subgroup $i(H')\subset S^2$,
        we can take 
        \begin{align*}
            (f, -f),\; &(e_2, 0),\; (e_3, 0),\; (e_4, 0),\;\\
            (0, e_1),\; &(0, e_2),\; (0, e_3),\; (0, e_4)
        \end{align*}
        as a basis for $\Gamma$.
        Then we readily see that
        \begin{align*}
            \Gamma_0
            &=P_0\cap\Gamma\\
            &=\mathbb{Z}\langle\,(f, -f),\; (e_2, -e_2),\; (e_3, -e_3),\; (e_4, -e_4)\,\rangle.
        \end{align*}
        and therefore, 
        \begin{align*}
            \phi(\Gamma_0)
            &=\mathbb{Z}\langle\,(f, -f+e_2),\; (e_2, e_1),\; (e_3, -e_4),\; (e_4, e_3)\,\rangle.
        \end{align*}
        A direct computation shows that
        \begin{align*}
            &\Gamma/(\Gamma_0\oplus\phi(\Gamma_0))\\
            =&\frac{\mathbb{Z}\langle\,(f, -f),\; (e_2, 0),\; (0, e_1),\; (0, e_2)\,\rangle}{\mathbb{Z}\langle\,(f, -f),\; (e_2, -e_2),\; (f, -f+e_2),\; (e_2, e_1)\,\rangle}\\
            &\oplus \frac{\mathbb{Z}\langle\,(e_3, 0),\; (e_4, 0),\; (0, e_3),\; (0, e_4)\,\rangle}{\mathbb{Z}\langle\,(e_3, -e_3),\; (e_4, -e_4),\; (e_3, -e_4),\; (e_4, e_3)\,\rangle}\\
            =&\{1\}\oplus\Zn{2}\\
            =&\Zn{2} 
        \end{align*} 
        As a result, we have $\pi_1(X_\mathrm{reg})\cong\Zn{2}$ for the case of \ref{t4} and $m=1$.
    \end{proof}

    \begin{rem}
        \label{rem}
        Corollary \ref{tor} shows that $X_\mathrm{reg}$ is not simply connected in the cases of \ref{t3}, $m=0,1$ and \ref{t4}, $m=0,1$.
        Hence, $X$ is not an irreducible symplectic orbifold in these cases.
        The universal cover of $X_\mathrm{reg}$ extends to an \'{e}tale in codimension one cover $\tilde{X}\to X$ and $\tilde{X}$ is an irreducible symplectic orbifold (with only isolated quotient singularities).
        The following table shows the singularities of $\tilde{X}$ (we don't give a proof of this result here):
        \begin{center}
            \begin{tabular}{ccl}
                \toprule
                Case & $H$ & Singularities of $\tilde{X}$ \\
                \midrule
                \ref{t3}, $m=0$ & $\Zn{3}$ & smooth \\
                \ref{t3}, $m=1$ & $(\Zn{3})^2$ & smooth \\
                \ref{t4}, $m=0$ & $\Zn{4}$ & $a_2=36$ \\
                \ref{t4}, $m=1$ & $\Zn{2}\times\Zn{4}$ & $a_2=36$ \\ 
                \bottomrule
            \end{tabular}
        \end{center}
        For the first two cases, $\tilde{X}\cong \mathrm{Kum}^2(S)$.
        For the last two cases, $\tilde{X}$ is deformation equivalent to the Fujiki's example of case \ref{t4}, $m=2$.
        
    \end{rem}

    \bibliography{pi1}

\begin{thebibliography}{1}

\bibitem{Be83}
Arnaud Beauville.
\newblock Vari\'{e}t\'{e}s {K}\"{a}hleriennes dont la premi\`ere classe de
  {C}hern est nulle.
\newblock {\em J. Differential Geom.}, 18(4):755--782, 1983.

\bibitem{Mau}
Valeria Bertini, Annalisa Grossi, Mirko Mauri, and Enrica Mazzon.
\newblock Terminalizations of quotients of compact hyperk\"ahler manifolds by
  induced symplectic automorphisms, 2024.

\bibitem{BCHM}
Caucher Birkar, Paolo Cascini, Christopher~D. Hacon, and James McKernan.
\newblock Existence of minimal models for varieties of log general type.
\newblock {\em J. Amer. Math. Soc.}, 23(2):405--468, 2010.

\bibitem{Ca04}
Fr\'{e}d\'{e}ric Campana.
\newblock Orbifoldes \`a premi\`ere classe de {C}hern nulle.
\newblock In {\em The {F}ano {C}onference}, pages 339--351. Univ. Torino,
  Turin, 2004.

\bibitem{Fu82}
Akira Fujiki.
\newblock On primitively symplectic compact {K}\"{a}hler {$V$}-manifolds of
  dimension four.
\newblock In {\em Classification of algebraic and analytic manifolds ({K}atata,
  1982)}, volume~39 of {\em Progr. Math.}, pages 71--250. Birkh\"{a}user
  Boston, Boston, MA, 1983.

\bibitem{Kol93}
J\'{a}nos Koll\'{a}r.
\newblock Shafarevich maps and plurigenera of algebraic varieties.
\newblock {\em Invent. Math.}, 113(1):177--215, 1993.

\bibitem{menet}
Grégoire Menet.
\newblock Thirty-three deformation classes of compact hyperk\"ahler orbifolds,
  2022.

\bibitem{Na22}
Yoshinori Namikawa.
\newblock Birational geometry for the covering of a nilpotent orbit closure.
\newblock {\em Selecta Math. (N.S.)}, 28(4):Paper No. 75, 59, 2022.

\bibitem{Pe20}
Arvid Perego.
\newblock Examples of irreducible symplectic varieties.
\newblock In {\em Birational geometry and moduli spaces}, volume~39 of {\em
  Springer INdAM Ser.}, pages 151--172. Springer, Cham, [2020] \copyright 2020.

\end{thebibliography}
    \bibliographystyle{plain}

\end{document}